\documentclass{amsart}

\usepackage{a4wide,latexsym,amssymb,amsthm,amsmath,color}
\usepackage{mathtools}%for prescripts (lower, upper index)

\theoremstyle{plain}

\newtheorem{theorem}{Theorem}

\newtheorem{proposition}{Proposition}

\theoremstyle{remark}

\title[Warped product immersions]
{ Lagrangian warped product immersions in $\mathbb{S}^6$ }

\author[Moruz]{Marilena Moruz}
\address{KU Leuven, Department of Mathematics, Celestijnenlaan 200B -- Box 2400, BE-3001 Leuven, Belgium}
\email{marilena.moruz@kuleuven.be}
\thanks{The author is a postdoctoral fellow of FWO - Flanders, Belgium.}

\begin{document}

\begin{abstract}
We study Lagrangian immersions in the nearly K\"ahler $\mathbb{S}^6$ which are  warped product manifolds of a $1$-dimensional base and a surface. Apart from the totally geodesic ones, they are either of constant sectional curvature $\frac{1}{16}$ or they satisfy equality in Chen's inequality, in which case the immersion is given explicitly.
%and study hypersurfaces in $\mathbb R^{n+1}$ which are a warped product
%of  $1$-dimensional base with an $(n-1)$-manifold of constant sectional curvature.
\end{abstract}

\maketitle

\section{Introduction}

The interest in the study of $\mathbb{S}^6$ dates back to 1955 with the results of  A. Fr\"olicher (\cite{pentrus6})  and Fukami and Isihara (\cite{fukami}), where an almost Hermitian structure was presented on $\mathbb{S}^6$ and which led to showing that  $\mathbb{S}^6$ is actually a nearly K\"ahler manifold. More recent results of J.-B. Butruille (\cite{butruille}) and P.-A. Nagy (\cite{nagy}) strengthened the importance of the nearly K\"ahler $\mathbb{S}^6$. It appears as one of the four classes of homogeneous nearly K\"ahler manifolds of dimension $6$, in the context of the structure theorem of  P.-A. Nagy.

The submanifolds of $\mathbb{S}^6$ have been actively investigated for many years afterwards by different authors, with interest in various types of properties (see for example \cite{onS6_6, onS6_1, onS6_2, onS6_5, onS6_3, Luc, onS6_4}).  A submanifold $M$ of $\mathbb{S}^6$ is called almost complex if its tangent space at a point $p$ stays tangent under the action of the almost complex structure $J$. Notice that the almost complex submanifolds must have even dimension. Moreover, Podest\`a and Spiro (\cite{podesta}) proved that the nearly K\"ahler manifolds of dimension $6$ do not admit $4$-dimensional almost complex submanifolds. Otherwise, if $J(T_pM)\subset T_p^{\perp}M$, then $M$ is called totally real. One may notice that the dimension of a totally real submanifold of $\mathbb{S}^6$ can be at most $3$. In the special case when the dimension is precisely $3$, then $M$ is called a Lagrangian submanifold and hence, $J$ interchanges the tangent and the normal spaces of $M$, that is $J(T_pM)=T_p^{\perp}M$.\\
The study of Lagrangian submanifolds of the nearly K\"ahler $\mathbb{S}^6$ is an area of interest for many authors and gathers rich results. A fundamental result is the one of N. Ejiri (\cite{ejiri}) which states that Lagrangian submanifolds in the nearly K\"ahler $\mathbb{S}^6$ are always minimal and orientable. In submanifold theory one investigates mainly the interplay between intrinsic (i.e. properties which depend only on the submanifold and not on the immersion) and extrinsic properties (i.e. properties depending on the immersion). As such there are two different types of questions which can be considered. Firstly, given a Riemannian manifold with certain geometric intrinsic properties one can ask whether it can be immersed as a Lagrangian submanifold and what can be said about the possible immersions. With this viewpoint, the first result was obtained by Ejiri, who showed (\cite{ejiri}) that a Lagrangian submanifold of $\mathbb{S}^6$ with constant sectional curvature is either totally geodesic or has constant sectional curvature $\frac{1}{16}$. However with this type of problem often some additional global conditions are necessary. For example, if the Lagrangian submanifold is compact and has sectional curvature $K>\frac{1}{16}$, then it is totally geodesic (\cite{opozda}). The work of K. Mashimo (\cite{mashimo}) classified the $G_2$-equivariant Lagrangian submanifolds of $\mathbb{S}^6$ into five models.\\
For the second type of questions, one typically assumes some conditions on the second fundamental form (extrinsic properties) and then one tries to determine all immersions satisfying these properties. For Lagrangian submanifolds of dimension $3$, many of those questions can be posed within the framework of the paper of Bryant (\cite{Bryant}). Using the properties of the totally symmetric cubic form $<h(X,Y),JZ>$ besides a generic Lagrangian submanifold, one can introduce $6$ additional special classes admitting some symmetries. For the special case of the nearly K\"ahler $6$ sphere, these classes were studied in \cite{onS6_1}, \cite{deszcz}, \cite{onS6_4}, \cite{onS6_5}. 
Note that all  the examples in the previous results are obtained starting from almost complex surfaces in $\mathbb{S}^6$ (of real dimension $2$) or from holomorphic surfaces in $\mathbb CP^2(4)$.  

In the present paper we want to look at another question of the first type. Namely we want to generalise the result of Ejiri to a larger class of Riemannian manifolds without imposing any global assumptions.  
The class we will investigate are the so called warped product manifolds. As it is well known such manifolds play a very  important role in physics (\cite{beem}, \cite{oneil2}). For example, the Schwarzchild space-time relativistic model, which describes the space around a massive star or black hole, is given as a warped product (\cite{oneil2}, p. 364-367). In fact, the warped product also arises naturally  in string theory, in relation with the classical instability of extra spatial dimensions. In \cite{67Nastasia}, Penrose claims that the spacetime, described in the theory as a direct product of certain manifolds, is physically unstable under small perturbations. From a mathematical point of view, in order to deal with perturbations in this context, one would rather consider a warped product manifold, than just a direct product manifold (\cite{Nastasia}).\\
However, also in submanifold theory many of the important examples admit a warped product structure. This can be seen, for instance, in the recent book of Chen (\cite{chen}) about the role of warped product in submanifold theory. Note that all rotational hypersurfaces in real space forms are warped product manifolds. The same is true for $H$-umbilical Lagrangian submanifolds of complex space forms. 

The present paper inquires on the case when the Lagrangian submanifold $M$ of the $6$-sphere is of warped product type, with a specific form: $M=I\times_fN$, where $I\subset \mathbb{R}$ and $N$ is a surface. We show that these submanifolds have constant sectional curvature or they are described by the following result.
\begin{theorem}
	Let $f:N^2\longrightarrow \mathbb{S}^6(1)$ be a minimal (non-totally geodesic) totally real immersion in $\mathbb{S}^6$ whose ellipse of curvature is a circle. Then $N^2$ is linearly full in a totally geodesic $\mathbb{S}^5$. Let $\vec{n}$ be a unit vector perpendicular to this $\mathbb{S}^5$. Then 
	\begin{align*}
	x:\left( -\frac{\pi}{2}, \frac{\pi}{2} \right)\times N^2\longrightarrow \mathbb S^6(1)\\
	(t,p)\mapsto \sin(t) \vec n+\cos(t)\tilde f(p)
	\end{align*}
	is  locally a totally real immersion which is warped product. Conversely, every locally warped product Lagrangian immersion, which is not of constant sectional curvature, can be obtained in this way.
	
\end{theorem}
The method we use is similar to the one developed in \cite{tsinghuapaper} for studying affine hypersurfaces with constant sectional curvature. In that paper, \cite{tsinghuapaper}, the method used, even though elementary, is called the Tsinghua principle. The same technique was also used in \cite{dioosvranckenwang} for studying Lagrangian submanifolds of the nearly K\"ahler $\mathbb{S}^3 \times \mathbb{S}^3$ with constant sectional curvature.

\section{Preliminaries}
\textbf{The nearly K\"ahler structure of $\mathbb{S}^6$.} 
The following part gives a short description of the nearly K\"ahler structure of $\mathbb{S}^6$ (see \cite{DVV}). Further references to this subject may be found in various papers, see for example  \cite{ejiri} or \cite{fukami}.\\
Let $e_0,e_1,\ldots,e_7$ denote the standard basis of  $\mathbb{R}^8$. A point 
$\alpha\in \mathbb{R}^8 $ can be written in a unique way as $\alpha = A e_0 +x$, for some $A\in\mathbb{R}$ and  $x\in span\{e_1,\ldots, e_7\}$. Therefore, $\alpha$ can be viewed as a Cayley number. We call $\alpha$ purely imaginary if  $A=0$. For $x,y$ two purely imaginary Cayley numbers we define the multiplication $ ``\cdot"$ given by: $x\cdot y= \langle x,y \rangle e_0+x\times y$, where $``\times"$ is defined according to the following multiplication table.
\begin{center}
	\begin{tabular}{ c| ccccccc }
		\hline
		$e_j\times e_k$ & 1 & 2 & 3 & 4 & 5 & 6 & 7 \\ \hline
		1 & $0$ & $e_3$ & $-e_2$ & $e_5$ & $-e_4$ & $e_7$ &$-e_6$  \\
		2 & $-e_3$ & $0$ & $e_1$ & $e_6$ & $-e_7 $ & $-e_4 $ & $e_5$  \\
		3 & $e_2 $ & $-e_1$ & $0$ & $-e_7$ & $-e_6$ & $e_5$ & $e_4$  \\
		4 & $-e_5$ & $-e_6$ & $e_7 $ & $0$ & $e_1 $ &$e_2$& $-e_3$  \\
		5 & $e_4$ & $e_7$ & $e_6$ & $-e_1$ & $0$ &$-e_3$ & $-e_2$  \\
		6 & $-e_7$ & $e_4$ & $-e_5$ &$-e_2$ &$e_3$ & $0$ & $e_1$  \\
		7 & $e_6$ & $-e_5$ & $-e_4$ & $e_3$ & $e_2$ & $-e_1$ & $0$  \\
	\end{tabular}
\end{center}
%Therefore, we may define the multiplication of two arbitrary Cayley numbers $\alpha = A e_0 +x$ and $\beta = B e_0 +y$ as $\alpha\cdot\beta=ABe_0+Ay+Bx+x\cdot y$, which gives $( \mathbb{R}^8, \cdot )$ a  Cayley algebra structure, denoted by $\mathcal{C}$.
Of course, the multiplication	$``\cdot"$  is neither commutative, nor associative. Let $\mathcal{C}_+$ denote the set of all purely imaginary Cayley numbers. We consider the set  
$$\mathbb{S}^6(1)=\{x\in \mathcal{C}_+ | \langle x,x \rangle =1\},$$ which defines the unit sphere of dimension $6$ centered at the origin. The tangent space $T_p\mathbb{S}^6 $ of $\mathbb{S}^6$  is given by  the subspace of $\mathcal{C_+}$ which is orthogonal to $p\in \mathbb{S}^6$. We define next an almost complex structure $J$ on $T_p\mathbb{S}^6 $ ($J$ is an  endomorphism such that $J^2=-Id$ ) as
$$ 	J_pU=p\times U, $$
where $p\in \mathbb{S}^6(1)$, $U\in T_p\mathbb{S}^6.$\\
Notice that the compact Lie group $G_2$ is the group of automorphisms of $\mathcal{C}$. It acts transitively on $\mathbb{S}^6(1)$ and preserves both $J$ and the standard metric on $\mathbb{S}^6(1)$.\\
Next, one defines $G(X,Y):=(\tilde{\nabla}_XJ)Y,$ for $X,Y\in \mathfrak{X}(\mathbb{S}^6)$ and $\tilde \nabla$ Levi-Civita connection on $\mathbb{S}^6$. One may check that $ G(X,X)=0$ holds (i.e. $ (\tilde \nabla_X J)X=0$) and therefore we obtain that the  almost complex structure $J$ defines a nearly K\"ahler structure on $(\mathbb{S}^6, \langle \cdot, \cdot\rangle, J)$.\\

\textbf{Lagrangian submanifolds of $\mathbb{S}^6$.} Let $M$ be a Lagrangian submanifold of $\mathbb{S}^6$ and $X,Y,Z,W$ tangent and $\xi$ normal vector fields to $M$. The Gauss and Weingarten formulas write out as
\begin{itemize}
	\item[(G)] $ \tilde{\nabla}_XY=\nabla_XY + h(X,Y) $, 
	\item[(W)] $\tilde{\nabla}_X\xi=-A_{\xi}X+\nabla^{\perp}\xi,$ 
\end{itemize} 
where $h$ is the second fundamental form of $M$  and $A_{\xi}$ is the shape operator of $\xi$.  Notice that they are related by $\langle h(X,Y),\xi \rangle=\langle A_{\xi}X,Y \rangle$ and we have as well that 
\begin{equation*}
A_{JY}X=-Jh(X,Y), \quad  \nabla^{\perp}_XJY=G(X,Y)+J\nabla_XY.
\end{equation*}
The equations of Gauss, Codazzi and Ricci are 
\begin{align}
R(X,Y)Z&=\langle Y,Z \rangle X-\langle X,Z \rangle Y  + [A_{JX},A_{JY}]Z,\label{gauss}\\
%R(X,Y)Z&=(\tilde{R}(X,Y)Z)^t+ A_{h(Y,Z)}X-A_{h(X,Z)}Y,\label{gausshyp}\\
(\nabla h)(X,Y,Z)&=(\nabla h)(Y,X,Z)\label{cdz2},\\
R^{\perp}(X,Y)JZ &=J [A_{JX},A_{JY}]Z.
\end{align} 
We recall that the normal curvature tensor $R^{\perp}(\cdot,\cdot)\cdot$ is defined as
\begin{equation*}
R^{\perp}(X,Y)\xi=\nabla^{\perp}_X\nabla^{\perp}_Y\xi-\nabla^{\perp}_Y\nabla^{\perp}_X\xi-\nabla^{\perp}_{[X,Y]}\xi
\end{equation*}
and the covariant derivative of $h$ is given by
\begin{equation*}
(\nabla h)(X,Y,Z)=\nabla^{\perp}_X h(Y,Z)-h(\nabla_XY,Z)-h(Y,\nabla_XZ).
\end{equation*}
The following Ricci identity holds
\begin{align}\label{ricci}
\begin{array}{l}
(\nabla^2 h)(X,Y,Z,W)-(\nabla^2 h)(Y,X,Z,W)=\\
R^{\perp}(X,Y)h(Z,W)-h(R(X,Y)Z,W)-h(Z,R(X,Y)W),
\end{array}
\end{align}
where the second covariant derivative of $h$ is defined as
\begin{align}\label{defsecder}
(\nabla^2 h)( X,Y,Z,W)=&\nabla^{\perp}_X(( \nabla h)(Y,Z,W))-(\nabla h)(\nabla _XY,Z,W)\\
& -(\nabla h)(Y,\nabla _XZ,W)-(\nabla h)(Y,Z,\nabla _XW ).\nonumber
\end{align}

\textbf{Warped product manifolds (\cite{chen}).} Let $B$ and $F$ be two Riemannian manifolds of positive dimensions, endowed with Riemannian  metrics $g_B$ and $g_F$, respectively. Let $f$ be a positive smooth function on $B$. On the product manifold $B\times F$, on which  we have the natural projections 
$$ \pi: B\times F\longrightarrow B \text{ and } \eta: B\times F\longrightarrow F,$$
we consider the Riemannian structure such that 
\begin{equation}\label{metric}
\langle X,Y\rangle=\langle \pi_{\star}(X),\pi_{\star}(Y)\rangle + f^2({\tiny\pi(x)}) g( \eta_{\star}(X), \eta_{\star}(Y) ), 
\end{equation} for any tangent vectors $X,Y\in TM$. 
This way $B\times F$ becomes a warped product manifold $M=B\times_f F$, for which the metric $g$ satisfies $g=g_B+f^2g_F$.\\
The function $f$ is called the warping function of the warped product, $B$ is called the base and $F$, the fiber. If $f$ is constant, then the warped product $B\times_f F$ is called trivial. In this case $B\times_f F$ is the Riemannian product $B\times F_f$, where $F_f$ is the Riemannian manifold equipped with the metric $f^2 g_F$.\\
The leaves $B\times \{q\}=\eta^{-1}(q)$ and the fibers $\{p\}\times F=\pi^{-1}(p)$ are Riemannian submanifolds of $M$. Vectors tangent to leaves are called horizontal and those tangent to fibers are called vertical. Let $\mathcal{H}$ denote the orthogonal projection  of $T_{(p,q)}M$ onto its horizontal subspace  $T_{(p,q)}(B\times\{q\})$ and let $\mathcal{V}$ denote the projection  onto the vertical subspace  $T_{(p,q)}(\{p\}\times F)$. If $u\in T_pB$, $p\in B$ and $q\in F$, then the lift $\bar{u}$ of $u$ to $(p,q)$ is the unique vector in  $T_{(p,q)}M$ such that $\pi_{\star}(\bar{u})=u$. For a vector field $X\in \mathfrak{X}(B)$, the lift of $X$ to $M$ is the vector field $\bar{X}$ whose value at $(p,q)$ is the lift of $X_p$ to $(p,q)$.  The set of all horizontal lifts is denoted by $\mathcal{L}(B)$. Similarly, we denote by $\mathcal{L}(F)$ the set of all vertical lifts.

One may express the curvature tensor of the warped product manifold in terms of the warping function and the curvature tensors of the two components $B$ and $F$.  For a warped product manifold $M=B\times_fF$, we define the lift $\tilde T$ of a covariant tensor $T$ on $B$ to $M$ as the pullback $\pi^*(T)$ via the projection $\pi:M\longrightarrow B$. We will denote by $^{B}\!R$ and $^{F}\!R$  the lifts to $M$ of the curvature tensors of $B$ and $F$, respectively.
\begin{proposition}[\cite{chen}]\label{prop}
	Let  $M=B\times_fF$ be a warped product of two (pseudo-)Riemannian manifolds. If $X,Y,Z\in \mathcal{L}(B)$  and $U,V,W\in \mathcal{L}(F)$, then we have
	\begin{itemize}
		\item[(1)] $R(X,Y)Z\in \mathcal{L}(B)$ is the lift of ${}^{B}\!R(X,Y)Z$ on $B$;
		\item[(2)] $R(X,V)Y=\frac{H^f(X,Y)}{f}V$;
		\item[(3)] $R(X,Y)V= R(V,W)X=0$;
		\item[(4)] $R(X,V)W=-\frac{\langle V,W\rangle}{f}\nabla_X(\nabla f)$;
		\item[(5)] $R(V,W)U= {}^F\!R(V,W)U+\frac{\langle \nabla f, \nabla f \rangle}{f^2}\{\langle V,U \rangle W-\langle W,U \rangle V\}$,
	\end{itemize}
	where $R$ is the curvature tensor of $M$ and $H^f$ is the Hessian of $f$. 
	
\end{proposition}

\textbf{Chen's ideal submaniolds.}
In 1993 B.-Y. Chen introduced (\cite{pinching}) the Riemannian invariant $\delta_M: M\rightarrow \mathbb R$ defined on a Riemannian manifold $M$ of dimension $n$ as
$$
\delta_M(p)=\tau(p)-(inf\ K)(p),
$$
where
$$
(inf\ K)(p)=inf\{K(\pi)|\, \pi \text{ is a } 2\text{-dimensional subspace of } T_pM\},
$$
$K(\pi)$ is the sectional curvature of $\pi$ and $\tau(p)=\sum\limits_{i<j}K(e_i\wedge e_j)$ is the scalar curvature  with respect to an orthonormal basis $\{e_1,\ldots,e_n\}$ of the tangent space of $M$ at $p$.  Furthermore, using this invariant, the author proved the following sharp inequality for submanifolds $M$ of dimension $n$ of real space forms $\tilde M(\tilde c)$:
\begin{equation}\label{delta}
\delta_M\leq \frac{n^2(n-2)}{2(n-1)} \|H\|^2+\frac{1}{2}(n+1)(n-2)\tilde c,
\end{equation}
where $H$ is the mean curvature vector field of $M$ in $\tilde M$ and $\|H\|$ is the mean curvature function. B.-Y. Chen further investigated such submanifolds of dimension $3$ for which the above inequality becomes equality. In the case when the real space form is $\mathbb{S}^6$  and $M$ is a Lagrangian submanifold, the inequality in \eqref{delta} becomes $\delta_{M}\leq 2$, since Lagrangian submanifolds in $\mathbb{S}^6$ are minimal and of dimension $3$.

Deszcz, Dillen, Verstraelen and Vrancken (\cite{deszcz}) proved that the Lagrangian submanifolds which realize equality with $\delta_{M}=2$ are quasi-Einstein. In general, an $n$ dimensional manifold is called quasi-Einstein if its Ricci tensor has an eigenvalue of multiplicity at least $n-1$.

The study of Lagrangian submanifolds with $\delta_M=2$ was started in \cite{CDVV} and \cite{chen36}. 
The complete classification of Lagrangian submanifolds in the nearly K\"ahler $6$-sphere satisfying $\delta_M=2$ was established by Dillen and Vrancken (\cite{onS6_1}). These also correspond to the Lagrangian submanifolds admitting a $\mathbb{S}_3$ symmetry in the framework of Bryant.   \\
The previous inequality is relevant in the general context of submanifolds theory, when trying to  find the relationship between the main intrisic invariants and the main extrinsic invariants of a submanifold.\\

\section{Results}
We study  minimal Lagrangian warped product submanifolds isometrically immersed in $\mathbb{S}^6$, of the form
$$ M=I \times_f N\mapsto \mathbb{S}^{6}, $$
where  $I \subset \mathbb{R}$, $f$ is a function defined on $I$ and  $(N, g)$ is a surface. Note that we will sometimes use the notation $N^2$ to indicate that $N$ is $2$-dimensional.

Let $E_1=\frac{\partial}{\partial t}$ be the unit vector tangent to $I$ and $\tilde U,\tilde V$ tangent vectors to $N$. Then from \eqref{metric} it follows that:
$$
\langle E_1, E_1\rangle=1,\quad
\langle  E_1,\tilde U \rangle=0,\quad
\langle \tilde U,\tilde V\rangle=f^2 g(\tilde U,\tilde V).
$$
Moreover, from  Proposition \ref{prop} we have: 
\begin{align*}
R(E_1,\tilde V)E_1&=\frac{f''}{f}\tilde V,\\
R(\tilde W,\tilde V)E_1&=0,\\
R(E_1,\tilde V)\tilde W&=-\langle \tilde V,\tilde W \rangle \frac{f''}{f}E_1,\\
R(\tilde V,\tilde W)\tilde U&={}^{N}\!R(\tilde V,\tilde W)\tilde U+\frac{f'^2}{f^2}\{\langle \tilde V,\tilde W \rangle \tilde W-\langle \tilde W,\tilde U\rangle \tilde V\},
\end{align*}
where ${}^{N}\!R$ is the curvature tensor on $N$ and $\tilde W$ is tangent to $N$.

In order to study such Lagrangian submanifolds $M$ in $\mathbb{S}^6$, we will make use of the technique introduced  in \cite{tsinghuapaper}, as the Tsinghua principle. It is described in the computations below and one should notice that it leads to obtaining an additional formula (see equation \eqref{eq1}) involving the second fundamental form, which will eventually provide new information about the warped product submanifold. The advantage of this technique, compared with using directly the Gauss equation, is that the resulting equations are linear in the components of the second fundamental form (instead of quadratic).
It is worth mentioning that depending on the form of the Codazzi equation of the submanifold (notice that there is no normal part of the curvature tensor of the ambiant manifold in \eqref{cdz2}), the computations of the Tsinghua principle can sometimes turn out to be more lengthy or even more complicated. However, it is not the case here (notice the zero right-hand side term in \eqref{eq2}). A few other applications of this technique can be found in \cite{BartDioosPHDthesis, dioosvranckenwang}.\\ 
We proceed by taking the derivative with respect to a tangent vector $W\in T_{(p,q)}M$ in the Codazzi equation \eqref{cdz2} and we see that
\begin{equation}\label{eq2}
(\nabla^2h)(W,X,Y,Z)-(\nabla^2h)(W,Y,X,Z)=0.
\end{equation} In this relation we permute cyclically $W,X,Y$ and use again the Ricci identity in \eqref{ricci} together with the first Bianchi identity to obtain
\begin{align}\label{eq}
0=& R^{\perp}(W,X)h(Y,Z)-h(Y,R(W,X)Z)+\nonumber\\
&  R^{\perp}(X,Y)h(W,Z)-h(W,R(X,Y)Z)+\\
&  R^{\perp}(Y,W)h(X,Z)-h(X,R(Y,W)Z).\nonumber
\end{align}
From the equations of Gauss and Ricci we have that $R^{\perp}(X,Y)JZ = JR(X,Y)Z +\langle X,Z\rangle JY -\langle Y,Z \rangle JX$.
Therefore, equation \eqref{eq} becomes:
\begin{align}\label{eq1}
0=&JR(W,X)Jh(Y,Z)+ h(Y, R(W,X)Z)\nonumber\\
&JR(X,Y)Jh(W,Z)+ h(W,R(X,Y)Z)\\
&JR(Y,W)Jh(X,Z)+h(X, R(Y,W)Z).\nonumber
\end{align}
Next, we apply $J$ and take the inner product in \eqref{eq} with $ V\in T_{(p,q)}M $. We take $\tilde U,\tilde V$ to form an orthonormal basis on $T_qN$ and we choose  $Y=V=E_1$ and $X=Z=\tilde U$, $W=\tilde V$.
As $^{N}R(\tilde V, \tilde U)\tilde U=\frac{K_{N}(q)}{f^2}(\langle \tilde U,\tilde U \rangle \tilde V-\langle \tilde V, \tilde U \rangle \tilde U)$, where  $K_{N}(q)$ denotes the Gaussian curvature of $N$ at $q$,  we get that 
$$
(K_{N}(q)-f'^2+ff'')\langle A_{JE_1}E_1,\tilde V\rangle=0.
$$
Similarly, we obtain
$$
(K_{N}(q)-f'^2+ff'')\langle A_{JE_1}E_1,\tilde U\rangle=0.
$$
Hence, there are two cases:
\begin{itemize}
	\item[ (i.)] $K_{N}(q)-f'^2+ff''=0$;
	\item[(ii.)] $K_{N}(q)-f'^2+ff''\neq 0$.
\end{itemize}
%%%%%%%%%%%%%%%%%%%%5
Let us treat first the case when  $K_{N}(q)-f'^2+ff''=0$. By straightforward computations we obtain that the sectional curvature of $M$ is given by
$$
K_{M}(\pi)=-\frac{f''}{f}, \text{ for all }2\text{-planes }  \pi\in T_pM. 
$$
As $dim M=3$, it follows that $M$ has constant sectional curvature. By the results in \cite{ejiri} we know that  $M$ is either totally geodesic (i.e. it has sectional curvature equal to $1$) or  it has constant sectional curvature $\frac{1}{16}$ (in which case, $M$ is congruent to one of the homogeneous examples of Mashimo, see \cite{mashimo}).

In the case when $K_{N}(q)-f'^2+ff''\neq 0$ holds, it follows that $E_1$ is an eigenvector of $A_{JE_1}$. We denote by $E_2,E_3$ the other eigenvectors of $A_{JE_1}$, by $\mu_1, \mu_2,\mu_3$ the corresponding eigenvalues and obtain that 
$$Jh(E_1,E_k)=\mu_k E_k,\, k=1,2,3.$$
Let us now choose $W=E_1$, \, $X=Z=E_2$, \, $Y=V=E_3$ in \eqref{eq}. It follows that
\begin{align*}
0=&(\frac{f'^2}{f^2}-\frac{f''}{f})\langle E_2, E_2\rangle \langle E_3, Jh(E_1,E_3)\rangle-  (\frac{f'^2}{f^2}-\frac{f''}{f})\langle E_3, E_3\rangle \langle E_2, Jh(E_1,E_2)\rangle\\
&+\langle {}^{N}\!R(E_2,E_3)E_2, Jh(E_1,E_3)\rangle +
\langle {}^{N}\!R(E_2,E_3)E_3, Jh(E_1,E_2)\rangle, 
\end{align*}
which implies 
\begin{equation*}
0=(\mu_3-\mu_2)(f'-f''f-K_{N}(q))
\end{equation*}
and, hence, $\mu_3-\mu_2=0$. Let $\mu:=\mu_2=\mu_3$. By straightforward computations, we obtain that there exist local functions $a,b,c,d$ such that 
\begin{equation}\label{Luc}
\begin{matrix*}[l]
Jh(E_1,E_1)=\mu_1 E_1,& Jh(E_2,E_2)=\mu E_1+a E_2+d E_3,\\
Jh(E_1,E_2)=\mu E_2,  & Jh(E_2,E_3)=d E_2 +b E_3,\\
Jh(E_1,E_3)=\mu E_3,  & Jh(E_3,E_3)=\mu E_1+bE_2+c E_3.
\end{matrix*}
\end{equation}
From \cite{Luc} we know that we may choose $E_2, E_3$ such that \eqref{Luc} holds for some functions $a,b,c$ and $d=0$. Moreover, by minimality of $M$ we have that $\mu_1=-2\mu,$ $b=-a$ and $c=0$. Therefore,
\begin{equation}\label{Luc2}
\begin{matrix*}[l]
Jh(E_1,E_1)=-2 \mu E_1,& Jh(E_2,E_2)=\mu E_1+a E_2,\\
Jh(E_1,E_2)=\mu E_2,  & Jh(E_2,E_3)=-a E_3,\\
Jh(E_1,E_3)=\mu E_3,  & Jh(E_3,E_3)=\mu E_1-a E_2.
\end{matrix*}
\end{equation}
Either by straightforward computations or by using Lemma 3.1. in \cite{deszcz}, we see that the Ricci tensor of $M$ has one eigenvalue of multiplicity $2$ and so $M$ is, indeed, quasi-Einstein.\\
If $M$ is not totally geodesic, we may distinguish two cases, as $\mu=0$ or $\mu\neq 0$.\\
%%%case
\textbf{Case 1.} $\mu=0$.
We obtain that 
\begin{equation}\label{star}
\begin{matrix*}[l]
Jh(E_1,E_1)=a E_1,& Jh(E_2,E_2)=-a E_1,\\
Jh(E_1,E_2)=-a E_2,  & Jh(E_2,E_3)=0,\\
Jh(E_1,E_3)=0,  & Jh(E_3,E_3)=0,
\end{matrix*}
\end{equation}
after having renamed the vectors in the basis $E_1,E_2,E_3\leadsto E_3,E_1,E_2$. Notice that $a \neq 0$,  $E_3$ is tangent to $I$ and  $E_2,E_3$ are tangent to $N$. At this point, we should refer the reader to \cite{CDVV}.
Given \eqref{star}, we see by Lemma 6.1 in \cite{CDVV} that $M$ satisfies equality in Chen's inequality ($\delta_{M}=2 $). Moreover, it is shown in \cite{pinching} that for a submanifold M of dimension $n$, with mean curvature vector $H$, the following distribution 
$$
\mathcal{D}=\{X\in T_pM| (n-1)h(X,Y)=n \langle X,Y \rangle H, \forall\ Y\in T_pM\}
$$
is integrable when it has constant dimension. In the case of our submanifold, we have $\mathcal{D}=\{E_3\}$ and therefore $\mathcal{D}^{\perp}=\{E_1,E_2 \}$. We see that the dimension of $D$ is constant and, because of the warped product condition, $\mathcal{D}^{\perp}$ is an integrable distribution.\\
% Indeed, from \cite{CDVV} we know that  $\nabla_{E_1}E_2=-\tilde c E_1+\tilde b E_3,\quad \nabla_{E_2}E_1=\tilde d E_2-\tilde  bE_3$, which implies that $\tilde b=0$. \\

\noindent\textbf{Case 2.} $ \mu\neq 0$. 
We have that \eqref{Luc2} holds for $a\neq 0$.
Moreover, the two distributions $\mathcal{D}$ and $\mathcal{D}^{\perp}$ are defined as $\mathcal{D}=\{E_1\}$ (tangent to $I$) and $\mathcal{D}^{\perp}=\{E_2,E_3\}$ (tangent to $N$). As $M$ is a warped product we know that $\mathcal{D}^{\perp}$ is integrable. However, from the proof of Theorem 2 in \cite{deszcz} we may see that for a Lagrangian submanifold with the second fundamental form as given by \eqref{Luc2}, the Lie bracket of $E_2, E_3$ always has a non vanishing  component in the direction of $E_1$. Therefore this case doesn't hold.\\

\noindent To conclude, the following result describes the immersion of $M$ into $\mathbb{S}^6$.
The starting point of the construction is a totally real immersion of a $2$-dimensional surface in $\mathbb S^6$ with ellipse of curvature a circle. It is shown in Theorem 5.3 of \cite{boltonvranckenwoodward} that such a surface is contained in a totally geodesic $\mathbb S^5$ and by Theorem 7.2 of \cite{boltonvranckenwoodward} that it can be constructed from a minimal Lagrangian surface in $\mathbb CP^2(4)$ using an appropriate Hopf lift. 

\begin{theorem}[ \cite{CDVV} ]\label{teoremrezult}
	Let $\tilde f:N^2 \longrightarrow \mathbb{S}^6(1)$ be a minimal (non-totally geodesic) totally real immersion in $\mathbb{S}^6$ whose ellipse of curvature is a circle. Then $N^2$ is  linearly full in a totally geodesic $\mathbb S^5$. Let $\vec n$ be a unit vector perpendicular to this $\mathbb S^5$. Then 
	\begin{align*}
	x:\left( -\frac{\pi}{2}, \frac{\pi}{2} \right)\times N^2\longrightarrow \mathbb S^6(1)\\
	(t,p)\mapsto \sin(t) \vec n+\cos(t)\tilde f(p)
	\end{align*}
	is a totally real immersion which satisfies the equality in $\delta_M\leq 2$.
	Conversely, every totally real (non-totally geodesic) immersion of $M$ into $\mathbb S^6(1)$ satisfying
	\begin{enumerate} 
		\item $\delta_M=2$,
		\item the dimension of $\mathcal{D}$ is constant,
		\item $\mathcal{D}^\perp$ is an integrable distribution
	\end{enumerate}
	can be locally obtained in this way.
\end{theorem}

\begin{proposition}
	The induced metric by the immersion $x$ on $( -\frac{\pi}{2}, \frac{\pi}{2} )\times N^2$ is a warped product metric.
\end{proposition}
\begin{proof}
	Let $t,u,v$ be local coordinates on $( -\frac{\pi}{2}, \frac{\pi}{2} )\times N^2$. As the normal to the totally geodesic $\mathbb{S}^5$ in $\mathbb{S}^6$ is a constant vector, we obtain that 
	$$
	x_t=\cos(t)\vec{n}-\sin(t)\tilde f,\quad x_u=\cos(t)\tilde f_u,\quad x_v=\cos(t)\tilde f_v.
	$$
	We immediately obtain that the induced metric is given by 
	$$
	\langle x_u,x_u \rangle=\cos^2(t) g(\tilde f_u,\tilde f_u), \quad \langle x_v,x_v \rangle=\cos^2(t) g(\tilde f_v,\tilde f_v),
	$$  
	$$
	\langle x_u,x_v \rangle=\cos^2(t) g(\tilde f_u,\tilde f_v), \quad \langle x_t,x_u \rangle=\langle x_t,x_v \rangle=0.
	$$
	This concludes the proof of the main theorem.
\end{proof}

%\begin{theorem}\label{t1}
%	Let $I \times_f N^2$ be a warped product manifold and $i:I \times_f N^2\rightarrow \mathbb{S}^6 $  be an isometric Lagrangian immersion in $\mathbb{S}^6$. Then $i$ is either totally geodesic or has constant sectional curvature $\frac{1}{16}$ or is locally obtained as in  Theorem \ref{teoremrezult}.
%\end{theorem}

\end{document}